\documentclass[12pt]{amsart}%
\usepackage{graphicx}
\usepackage{amscd}
\usepackage{geometry}
\usepackage{tikz-cd}
\usepackage{amsmath}
\usepackage{amsfonts}
\usepackage{float}
\usepackage{amssymb}%
\providecommand{\U}[1]{\protect\rule{.1in}{.1in}}
\newtheorem{theorem}{Theorem}[section]
\theoremstyle{plain}

\newtheorem{corollary}[theorem]{Corollary}

\newtheorem{lemma}[theorem]{Lemma}

\theoremstyle{definition}
\newtheorem{definition}[theorem]{Definition}
\newtheorem{remark}[theorem]{Remark}
\newtheorem{example}[theorem]{Example}
\numberwithin{equation}{section}

\makeatletter
\renewcommand*\env@matrix[1][*\c@MaxMatrixCols c]{%
	\hskip -\arraycolsep
	\let\@ifnextchar\new@ifnextchar
	\array{#1}}
\makeatother

\begin{document}
\title{The Jones Polynomial of a Connect Sum is Multiplicative: A New Approach Via Trip Matrices}

\author{Molly A. Moran}
\address{Department of Mathematics, The Colorado College, Colorado Springs, Colorado 80903}
\email{mmoran@coloradocollege.edu}
\author{Emerson Worrell}
\address{Department of Mathematics, Oregon State University, Corvallis, OR 97331}
\email{worrelle@oregonstate.edu}

\date{}

\begin{abstract}
	\noindent We utilize the trip matrix method of calculating the Jones Polynomial to give an alternative proof that the Jones Polynomial is multiplicative under connect sums. 
	
	\noindent\textbf{AMS classification numbers}. 57K10, 57K14
\end{abstract}
\maketitle

\section{Introduction}
\label{sec: introduction and motivation}

The Jones polynomial of a knot $K$, denoted $V_K$, was introduced in 1984 by V. Jones \cite{Jones1} as a new polynomial knot invariant. Much is known about this invariant, but, like many areas in knot theory, there are still many open questions. There are also calls for finding more elementary proofs of known results (see \cite{Adams}).  
 
 The trip matrix method, introduced by Louis Zulli \cite{Zulli}, provides an alternative method of computing the Jones Polynomial using the state method \cite{Kauffman2} and basic linear algebra. Given a knot diagram with $n$ crossings, the trip matrix is an $n\times n$ matrix over $\mathbb{Z}_2$ that encodes information about the crossings that one encounters while traveling along the knot. We use the trip matrix to provide an alternative, elementary proof that the Jones Polynomial is multiplicative over connect sums \cite{Jones2}. 

 In Section 2, we provide the necessary background information on trip matrices necessary for our work. In Section 3, we give results on the structure of trip matrices of composite knots which leads to a new proof of the multiplicative nature of the Jones Polynomial. 

\section{The Trip Matrix}
\label{sec: background information}

We begin with the construction of the trip matrix as described in \cite{Zulli}. Let $K$ be a knot diagram with $n$ crossings. Adorn the knot diagram by first labeling the $n$ crossings $1$ through $n$ in any order. At each over-crossing, choose a direction along the knot at random and draw an arrow pointing in that direction. At each under-crossing, place an arrow that is oriented to point counterclockwise from its corresponding over-crossing arrow. From this, we construct a matrix corresponding to this knot diagram, called the trip matrix, which we denote by $T_K$.
	
Each entry in the $i^{th}$ row and $j^{th}$ of $T_K$ will be either $0$ or $1$ depending on the following rules:  
	\begin{itemize}
		\item If $i=j$, follow the over-crossing arrow at crossing $i$ along the knot diagram in the direction indicated by the arrow until you return to crossing $i$ at the under-crossing arrow. If the under-crossing arrow leads you on, that is, if it points in the direction you are currently traveling, then $T_{ii}=0$. If the under-crossing arrow pushes you back, that is, if the arrow points to where you are traveling from, then $T_{ii}=1$. 
	
		\item If $i\neq j$, follow the over-crossing arrow at crossing $j$. Count the number of times modulo $2$ that you pass through crossing $i$ before reaching crossing $j$ again. This number is the $T_{ij}$ entry. 
	\end{itemize} 
    
We note that the trip matrix is inherently symmetric by construction. Therefore, we may begin at either crossing when computing $T_{ij}$, and this also reduces the amount of computation needed. 

\begin{example}
	 Consider an adorned diagram of the figure eight knot $K$ in Figure \ref{fig: Figure 8}. 
	
	\begin{figure}[h]
		\centering
		\includegraphics[width = .4\textwidth]{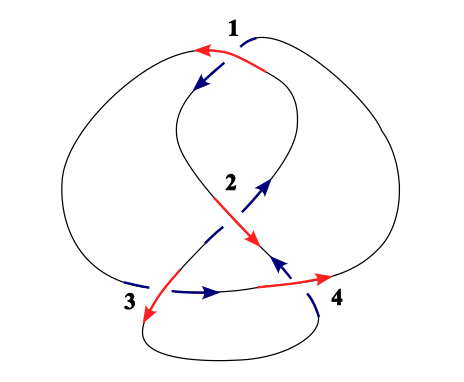}
	\caption{The Figure 8 Knot.}
    \label{fig: Figure 8}
	\end{figure}
    
As the diagram has four crossings, $T_K$ will be a $4 \times 4$ matrix over $\mathbb{Z}_2$. We compute the diagonal entries first. Starting at crossing $1$, we follow the over-crossing arrow and when we return to the under-crossing, the arrow leads us on, so $T_{11}=0$. Crossing $2$ behaves similarly and hence $T_{22}=0$. For crossings $3$ and $4$, the under-crossing arrows leads back the way we came, so $T_{33}=T_{44}=1$.
	
For the non-diagonal entries, we first observe that if we begin at the over-crossing labeled $1$ and trace the knot, we pass through crossing $3$, crossing $4$, and then arrive back at crossing $1$. Hence, $T_{13}=T_{31}=1$, $T_{14}=T_{41}=1$, and $T_{12}=T_{21}=0$. Repeating this process at over-crossings $2,3$ and $4$, we obtain the following trip matrix:
	
	\[\begin{bmatrix}
	0 & 0 & 1 & 1 \\
	0 & 0 & 1 & 1 \\
	1 & 1 & 1 & 0 \\
	1 & 1 & 0 & 1 \\
	\end{bmatrix}\]
\end{example}

 A natural question to ask when first working with the trip matrix is what happens to the trip matrix if we adorn the knot differently. Firstly, the choice of orientation on the over-crossing will not affect the trip matrix \cite{Zulli}. However, the numbering of the crossings can generate different trip matrices. Consider, for example, the figure eight knot with two different labelings shown in Figure \ref{fig: different labellings}. 

\begin{figure}[h]
	\centering
	\includegraphics[width = .75\textwidth]{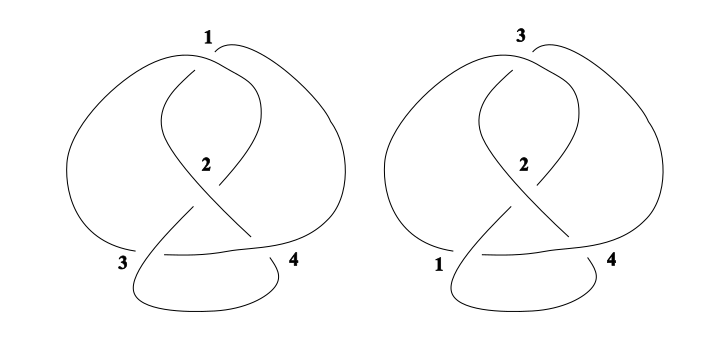}
	\caption{Different labelings for the figure 8 knot.}
	\label{fig: different labellings}
\end{figure}

Regardless of choice of orientation of arrows, the adorned knots will have the following respective trip matrices: 
\[\begin{bmatrix}
0 & 0 & 1 & 1 \\
0 & 0 & 1 & 1 \\
1 & 1 & 1 & 0 \\
1 & 1 & 0 & 1 
\end{bmatrix}
\qquad 
\begin{bmatrix}
1 & 1 & 1 & 0 \\
1 & 0 & 0 & 1 \\
1 & 0 & 0 & 1 \\
0 & 1 & 1 & 1 
\end{bmatrix}\]

Clearly, these two matrices are different. However, they are related as follows: in the first matrix, swap row $1$ and $3$, followed by a swap of columns $1$ and $3$. 
\[ \begin{bmatrix}
0 & 0 & 1 & 1 \\
0 & 0 & 1 & 1 \\
1 & 1 & 1 & 0 \\
1 & 1 & 0 & 1 
\end{bmatrix} \rightarrow 
\begin{bmatrix}
1 & 1 & 1 & 0 \\
0 & 0 & 1 & 1 \\
0 & 0 & 1 & 1 \\
1 & 1 & 0 & 1 
\end{bmatrix} \rightarrow
\begin{bmatrix}
1 & 1 & 1 & 0 \\
1 & 0 & 0 & 1 \\
1 & 0 & 0 & 1 \\
0 & 1 & 1 & 1 
\end{bmatrix}
\]

Notice that this operation transforms the first matrix into the second. Choosing to swap rows and columns $1$ and $3$ was not random: in Figure \ref{fig: different labellings}, crossings $1$ and $3$ are swapped, and since the rows and columns in the trip matrix correspond to the labels on crossings, it is natural to swap these in the trip matrix. We now give this operation on a matrix a name and prove some basic properties.

\begin{definition}[Row-Column Swap]
	\label{def: row/column swapping}
	Given a square matrix $A$, the \textit{row-column swap operation}, denoted by $\Delta(i,j)$, is defined by first swapping row $i$ with row $j$ in the matrix $A$, followed by swapping column $i$ with column $j$.
\end{definition}
We first observe that the row-column swap operation is equivalent to a column-row swap:

\begin{lemma} \label{lemma: delta commutative} The $\Delta$ operation gives rise to the same matrix independent of whether the rows or columns are swapped first. \end{lemma} 
\begin{proof}
	Let $T$ be an $n\times n$ matrix and $i,j,k \in \{1,2,3,...n\}$ be distinct. Consider $\Delta(i,j)$. We begin by swapping rows first, then columns, observing where each entry in $T$ is sent; here, the first arrow denotes the swapping of rows $i$ and $j$ and the second columns $i$ and $j$: 
\[
\begin{array}{c|c|c}
 T_{ii} \rightarrow T_{ji} \rightarrow  T_{jj} & T_{ji} \rightarrow T_{ii} \rightarrow  T_{ij} &  T_{ik} \rightarrow T_{jk} \rightarrow  T_{jk}      \\
 & & \\
 T_{jj} \rightarrow T_{ij} \rightarrow  T_{ii} & T_{ij} \rightarrow T_{jj} \rightarrow  T_{ji} & T_{ki} \rightarrow T_{ki} \rightarrow  T_{kj}
     
\end{array}\]
	
We now repeat the process by first swapping columns, followed by rows:

\[
\begin{array}{c|c|c}
 T_{ii} \rightarrow T_{ij} \rightarrow  T_{jj} & T_{ij} \rightarrow T_{ii} \rightarrow  T_{ji} & T_{ik} \rightarrow T_{ik} \rightarrow  T_{jk} \\
 & & \\
 T_{jj} \rightarrow T_{ji} \rightarrow  T_{ii} & T_{ji} \rightarrow T_{jj} \rightarrow  T_{ij} & T_{ki} \rightarrow T_{kj} \rightarrow  T_{kj}
\end{array}
\]

The six entries tracked above are the only entries affected by $\Delta(i,j)$, since only entries in rows and columns $i$ and $j$ can be affected. We find that each entry is sent to the same place regardless of the order of operations, and thus $\Delta$ is independent of choice of swapping rows or columns first.
\end{proof}

\begin{theorem}
	\label{thm: delta = relabel}
	Given a diagram for a knot $K$ with trip matrix $T_K$, the $\Delta$ operation applied to $T_K$ is equivalent to the corresponding relabeling of the crossings in the knot diagram.  
\end{theorem}

\begin{proof}
	Let $T_K$ be the trip matrix for a knot $K$ under some labeling, and consider performing $\Delta(i,j)$ on $T_K$ to get a new matrix $T'_K$. In row $i$ and column $i$ of $T_K$, the relationship between crossing $i$ and the other crossings in the knot is recorded. After row-column swapping, these exact same relationships are cataloged into row and column $j$ as described in Lemma \ref{lemma: delta commutative}. Similarly the information pertaining to crossing $j$ is now located where the information regarding crossing $i$ previously was. This exact same result would be achieved if we were to instead swap the labels in $K$ of crossings $i$ and $j$; the same relationships would exist, but they would correspond to the rows and columns we swapped. Therefore, a swapping of two labels in the knot diagram is equivalent to performing the $\Delta$ operation on the corresponding rows and columns in the trip matrix.
\end{proof}

This relationship gives rise to a natural notion of equivalency between trip matrices, which we now define.

\begin{definition}[$\Delta$-equivalent]
	\label{delta equivalent}
	We say two matrices are \textit{$\Delta$-equivalent} if there exists some sequence of $\Delta$ operations that transforms one matrix into the other. If there is no such sequence of moves, then the two matrices are said to be \textit{$\Delta$-distinct}.
\end{definition}

For a given knot diagram $K$ with $n$-crossings, there are up to $n!$ possible trip matrices, corresponding to the $n!$ different choices of crossing labels. In some cases, the relabeling may not cause any changes to the trip matrix (see the trip matrix of the trefoil knot). However, if a relabeling does produce different trip matrices, Theorem \ref{thm: delta = relabel} guarantees these matrices are $\Delta$-equivalent.

\begin{remark}\label{rk:reidemeister} It is important to note that Theorem \ref{thm: delta = relabel} does not say that all trip matrices for a knot are $\Delta$-equivalent. A given knot has infinitely many different projections with different numbers of crossings, which lead to different trip matrices that cannot be $\Delta$-equivalent. For example, consider the figure eight knot from Figure \ref{fig: Figure 8}. In this projection, there are four crossings and hence, the trip matrix is $4\times 4$. If a simple twist is added to the knot, then the new projection has five crossings, which corresponds to a $5\times 5$ trip matrix. Clearly, this new trip matrix cannot be obtained from a simple row-column swap of a $4\times 4$ trip matrix. Thus, Theorem \ref{thm: delta = relabel} only applies to a fixed diagram for a knot $K$.  \end{remark}

\section{The Jones Polynomial of Connect Sums}
Given a trip matrix $T_K$, we can compute the Jones polynomial $V_K$ using properties of this matrix, as well as associated matrices that have been altered along the diagonal. To understand these modifications, we first define the state for a diagram. 

\begin{definition}[State] A \emph{state} $S$ for a diagram of a knot $K$ is a function that assigns to each crossing an $A$ or $B$. Let $\mathcal{S}(K)$ be the set of all possible states. 
\end{definition} 

We first observe that if a knot has $n$ crossings, there are $2^n$ possible states. We wish to assign a matrix to each state. The trip matrix $T_K$ corresponds to the state where each crossing has been assigned an $A$. If $S \in \mathcal{S}(K)$ is a state that is obtained from the all $A$ state by changing the crossings $i_1,i_2,...i_k$ to a $B$, we toggle (change from $0$ to $1$ or vice-versa) the entries $T_{i_1i_1},T_{i_2i_2},...,T_{i_ki_k}$ along the diagonal of $T_K$. We denote the resulting modified matrix by $T_{K_S}$. 

We use these toggled matrices to compute the Jones Polynomial as follows: 

\begin{theorem}\label{Zulli}
	\cite{Zulli}
	The \textit{Jones Polynomial} for a knot $K$ is given by the following formula:
\[V_K = (-t^{\frac{3}{4}})^{w(K)}\sum_{S \in \mathcal{S}(K)} t^{-\frac{1}{4}^{A(S)}} t^{\frac{1}{4}^{B(S)}}(-t^{-\frac{1}{2}}-t^{\frac{1}{2}})^{nul(T_{K_S})}\]
	where
	\begin{enumerate}
		\item 
		$w(K)$ is the number of ones on the diagonal minus the number of zeros on the diagonal of $T_K$,
		\item$A(S)$ and $B(S)$ are the number of $A$'s and $B$'s in state $S$, respectively, and 
		\item $nul(T_{K_S})$ is the dimension of the null space of $T_{K_S}$. 
	\end{enumerate}
\end{theorem}
In \cite{Jones2}, Jones proved that the Jones polynomial of a connect sum is the product of the Jones polynomials of the components. We provide an alternate elementary proof of this result using trip matrices. To begin, we examine the structure of the trip matrix and exploit this structure to calculate the quantities in Theorem \ref{Zulli}. 

\subsection{Trip Matrix for Connect Sums}

We begin with an example of the trip matrix for a composite knot constructed as the connect sum of 3 prime knots. Consider the knot $K$ with the adorned diagram in Figure 3. Observe that the three knot components of the  connect sum are clearly defined and the crossings are labeled in a way that respects the components.
	\begin{figure}[h]\label{fig:composite}
	\centering
	\includegraphics[width = \textwidth]{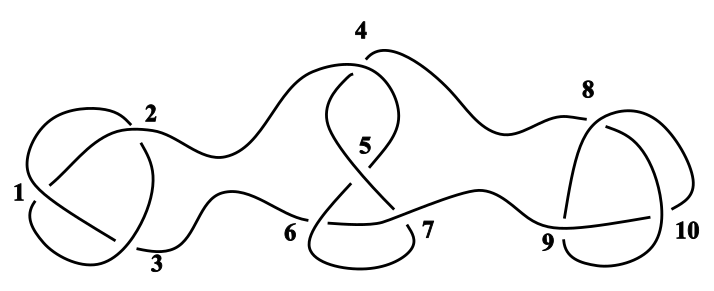}
	\caption{A composite knot with three factors}

\end{figure}
\begin{center}

\end{center}

Following the procedure presented in Section \ref{sec: background information}, the trip matrix $T_K$ with this adornment is
\[ T_K = \begin{bmatrix}[ccc|cccc|ccc]
1 & 1 & 1 & 0 & 0 & 0 & 0 & 0 & 0 & 0 \\
1 & 1 & 1 & 0 & 0 & 0 & 0 & 0 & 0 & 0 \\
1 & 1 & 1 & 0 & 0 & 0 & 0 & 0 & 0 & 0 \\
\hline
0 & 0 & 0 & 0 & 0 & 1 & 1 & 0 & 0 & 0 \\
0 & 0 & 0 & 0 & 0 & 1 & 1 & 0 & 0 & 0 \\
0 & 0 & 0 & 1 & 1 & 1 & 0 & 0 & 0 & 0 \\
0 & 0 & 0 & 1 & 1 & 0 & 1 & 0 & 0 & 0 \\
\hline
0 & 0 & 0 & 0 & 0 & 0 & 0 & 1 & 1 & 1 \\
0 & 0 & 0 & 0 & 0 & 0 & 0 & 1 & 1 & 1 \\
0 & 0 & 0 & 0 & 0 & 0 & 0 & 1 & 1 & 1 

\end{bmatrix}\]

We quickly observe that this matrix has a block structure. In particular, $T_K$ is a  \emph{block diagonal matrix}, an $n \times n$ matrix that is formed by placing smaller square matrices $A_1, A_2,...A_m$ along the diagonal with zeros in all remaining entries; in the above matrix, we have introduced lines to clearly denote the distinct blocks in the matrix. 

Every knot that arises from a connect sum of non-trivial prime knots has a knot diagram where each prime component is isolated as in Figure 3. If the crossings in this ``nice" diagram are numbered in groups respecting the components of the connect sum, the trip matrix will be block diagonal as we now prove. 

\begin{theorem}
	For every composite knot $K$, there is a trip matrix $T_K$ that is block diagonal where the blocks correspond to the trip matrices of the components of $K$. 
	\label{thm: composite implies block}
\end{theorem}
\begin{proof}
	Let $K= K_1\#K_2\#K_3...\#K_n$ where $K_j$ is a knot with $m_j$ crossings for each $j=1,2,...n$. Because $K$ is composite, it is possible to perform Reidemeister moves to obtain a knot diagram like the diagram in Figure 3, where the prime components $K_i$ are clearly and visibly independent from one another. Enumerate the crossings in $K$ such that the first $m_1$ crossings correspond to the $m_1$ crossings from $K_1$, crossings $m_1+1$ through $m_1+m_2$ correspond to the $m_2$ crossings from $K_2$, etc. 
    
    Suppose crossing $i$ is in component $K_j$. We first examine what happens to the entries in row $i$ that correspond to crossings in the component $K_m$ where $m\neq j$. Traveling in the direction of the over-crossing arrow at crossing $i$, there are two possibilities: either we leave the section of $K$ corresponding to $K_j$ before reaching crossing $i$ again, or we do not. If we do not leave the component, then the entries in row $i$ corresponding to the crossings in $K_m$ will all be zeros as we never pass through any crossings in the other components. If we do enter a section of $K$ corresponding to the component $K_m$, then we must necessarily pass through the entirety of $K_m$ before returning to crossing $i$ by construction of the connect sum. Hence, we pass through every crossing in $K_m$ twice, giving zeros in the trip matrix. Therefore, all entries in row $i$ that correspond to crossings not in the component $K_j$ are zeros, which forces a block diagonal structure. All that remains to check is that the blocks themselves are the trip matrices of the components. 

 To do so, we examine how crossing $i$ interacts with the other crossings in the $K_j$ component. If we stay in the same component, then we are following the path as if the component were not part of a connect sum. 
If we do enter a different component, the path will return in the exact same location relative to the other crossings as where it left, while traveling in the same direction. As a result the path relative to the other crossings in $K_j$ is the exact same as if it were an isolated knot, and so the entries corresponding to crossing $i$ from $K_j$ are identical to those in the trip matrix for $K_j$.\\

\end{proof}
In the proof of Theorem \ref{thm: composite implies block}, we were careful to label the crossings in a way that respected the components. We know that if we were to label the crossings differently, the matrix may no longer be block diagonal. However, Theorem \ref{thm: delta = relabel} guarantees that the matrix will be delta equivalent to a block diagonal matrix. Thus, combining Theorems \ref{thm: composite implies block} and \ref{thm: delta = relabel} we immediately obtain the following corollary. 

\begin{corollary} For every composite knot $K$, there is a matrix that is delta equivalent to a block diagonal matrix. 
\end{corollary}

\subsection{Properties Preserved In Trip Matrices of Connect Sums}

\noindent  From Theorem \ref{thm: composite implies block}, we know there is a trip matrix for a composite knot that is a block diagonal matrix with the blocks corresponding to trip matrices of the individual components. It would be natural to assume that many of the properties of the blocks are preserved in the large, composite matrix. In this subsection we show that all of the important information pertaining to the Jones Polynomial from Theorem \ref{Zulli} is preserved.

\begin{lemma}
	\label{thm: writhe preserved}
	Let $T_K$ be a block diagonal trip matrix for the knot $K$, and let $T_1, \dots ,T_n$ be the blocks in $T_K$. Then $w(T_K) = w(T_1) + \dots + w(T_n)$.
\end{lemma}
\begin{proof}
	Let $o(T_K)$ and $z(T_K)$ denote the number of ones and zeros along the diagonal of a given trip matrix $T$. By Theorem 3.2, $w(T_K) = o(T_K) - z(T_K)$. Also, 
    \[o(T_K) = o(T_1) + \dots + o(T_n)\]
    \[z(T_K) = z(T_1) + \dots + z(T_n)\]
    
   Thus,
    \[w(T_K) = o(T_1) + \dots + o(T_n) - (z(T_1) + \dots + z(T_n))\]
    \[= o(T_1) - z(T_1) + \dots o(T_n) - z(T_n) = w(T_1) + \dots + w(T_n)\]
\end{proof}

The following linear algebra result regarding block diagonal matrices is a simple application of the definition of the rank and nullity of a matrix. 

\begin{lemma}
	\label{thm: nullity preserved}
	Let $A$ be a block diagonal matrix with blocks $A_1$, $A_2$,...$A_m$. Then the following hold:
    
    \begin{itemize}
        \item $\text{rank}(A)=\text{rank}(A_1)+\text{rank}(A_2)+...+\text{rank}(A_m)$.
        \item $\text{nul}(A)=\text{nul}(A_1)+\text{nul}(A_2)+...+\text{nul}(A_m)$.
    \end{itemize}
\end{lemma}

We use the latter of these results to prove the following corollary, which applies this result specifically to the nullspace of toggled trip matrices:
\begin{corollary} Let $K$ be a knot with block diagonal trip matrix $T_K$. The nullity of a toggled matrix $T_{K_S}$ corresponding to state $S$ is the sum of the nullities of the toggled block matrices in $T_{K_S}$. 
\end{corollary}
\begin{proof} 
	We first observe that if $T_K$ is an $n\times n$ block diagonal matrix, then $T_{K_S}$ is also block diagonal, as $T_{K_S}$ is obtained from $T_K$ by only changing a subset of entries along the diagonal. Letting $T_1',T_2',...T_m'$ be the blocks in $T_{K_S}$, we find that by Theorem \ref{thm: nullity preserved}, we have

    \[\text{nul}(T_{K_S})=\text{nul}(T_1')+\text{nul}(T_2')+...+\text{nul}(T_m')\]
	\end{proof}

Finally, we examine how states are preserved in the block diagonal structure.
\begin{lemma}
	\label{thm: states preserved}
	Let $T_K$ be a block diagonal trip matrix for the composite knot $K=K_1\#K_2\#...\#K_n$, and let $T_1, \dots ,T_n$ be the respective trip matrix blocks corresponding to component knots $K_1, \dots ,K_n$. For a given state $S\in\mathcal{S}(K)$, $A(T_{K_S}) = A(T_{1_S}) + \dots + A(T_{n_S})$ and similarly, $B(T_{K_S}) = B(T_{1_S}) + \dots + B(T_{n_S})$. 
\end{lemma}
\begin{proof}
	A state for the knot $K$ can be thought of as word in the alphabet $\left\{A, B\right\}$ and can be subdivided into smaller words corresponding to the component knots $K_1, \dots ,K_n$.  Naturally, the total number of $A$'s and $B$'s in the large word will be the sum of the number of $A$'s and $B$'s in the smaller words. 
\end{proof}

\subsection{The Jones Polynomial is Multiplicative Over Connect Sums}
\label{sec: the jones polynomial is multiplicative over connect sums}

With all the pieces in place, we can now provide an alternative proof, first proved in \cite{Jones2}, that the Jones polynomial is multiplicative of over connect sums.

\begin{theorem}\cite{Jones2}
	\label{thm: Big Boy}
	The Jones Polynomial is multiplicative over connect sums. That is, if $K = K_1 \# K_2 \# \cdots \# K_n$, \[V_K = \prod_{i = 1}^n V_{K_i}\].
\end{theorem}
\begin{proof}
Suppose $K = K_1 \# K_2 \# \cdots \# K_n$. By Theorem \ref{thm: composite implies block}, there is a block diagonal trip matrix $T_K$ for $K$ where the blocks $T_1, \dots ,T_n$ are trip matrices for $K_1,K_2,...,K_n$, respectively. By Theorem \ref{Zulli}, the Jones Polynomials for the given knots are:  
	
	\[V_K = (-t^{\frac{3}{4}})^{w(K)}\sum_{S \in \mathcal{S}(K)} t^{-\frac{1}{4}^{A(S)}} t^{\frac{1}{4}^{B(S)}}(-t^{-\frac{1}{2}}-t^{\frac{1}{2}})^{nul(T_S)}\]
	
	\[V_{K_1} = (-t^{\frac{3}{4}})^{w(K_1)}\sum_{S \in \mathcal{S}(K_1)} t^{-\frac{1}{4}^{A(S)}} t^{\frac{1}{4}^{B(S)}}(-t^{-\frac{1}{2}}-t^{\frac{1}{2}})^{nul(T_{1_S})}\]
	
	\[\vdots\]
	
	\[V_{K_n} = (-t^{\frac{3}{4}})^{w(K_n)}\sum_{S \in \mathcal{S}(K_n)} t^{-\frac{1}{4}^{A(S)}} t^{\frac{1}{4}^{B(S)}}(-t^{-\frac{1}{2}}-t^{\frac{1}{2}})^{nul(T_{n_S})}\]
	
 Our claim is that $V_K = \prod_{i = 1}^n V_{K_i}$. 
	First, we verify the equality for terms outside of the summation. By Theorem \ref{thm: writhe preserved} we know that 
	
	\[w(K) = w(K_1) + \dots + w(K_n)\]
	
By exponent rules this means 
	
	\[(-t^{\frac{3}{4}})^{w(K)} = \prod_{i = 1}^n (-t^{\frac{3}{4}})^{w(K_i)}\]
	
	We move on to the pieces of the polynomial involving the summations. We first show that the product of the $n$ sums for the component knots has the same number of terms as the sum for $V_K$. Suppose $K$ has $m$ crossings so that $\mathcal{S}(K)$ has $2^m$ elements. Then there are $2^m$ terms in the unreduced sum for $V_K$. Suppose the components $K_1, K_2, \dots ,K_n$ have $m_1, m_2, \dots ,m_n$ crossings, respectively. Then each unreduced sum for $V_{{K_j}}$ will have $2^{m_j}$ terms. Thus, $\prod_{i = 1}^n V_{K_i}$ before reducing will have $2^{m_1} \times \dots \times 2^{m_n}$ terms in the sum. Since $m = m_1 + \dots + m_n$,  $2^m = 2^{m_1} \times \dots \times 2^{m_n}$; hence the number of terms in the unreduced sum for $V_K$ is the same as the number of terms in the unreduced sum for $\prod_{i = 1}^n V_{K_i}$.

 Because the number of terms on both sides of our alleged equality is the same, and because we can break up the states of the composite knot into smaller states in a very nice manner, our proof becomes very straightforward: we show that the polynomials are the same term-by-term. Consider a state $S \in \mathcal{S}(K)$. Then this state $S$ can be divided up into its component sub-states $S_1 \dots S_n$, where each sub-state $S_i$ is the section of the word $S$ that corresponds to the current state of the component knot $K_i$.\\
	
 Now consider the product of the terms from the sums for the component knots: 
	
	\[\prod_{i=1}^n t^{-\frac{1}{4}^{A(S_i)}} t^{\frac{1}{4}^{B(S_i)}}(-t^{-\frac{1}{2}}-t^{\frac{1}{2}})^{nul(T_{i_{S_i}})}\]
	
 We can break this product into three separate products for each like term in the expression: 
	
	\[\prod_{i=1}^n t^{-\frac{1}{4}^{A(S_i)}} \prod_{i=1}^n t^{\frac{1}{4}^{B(S_i)}} \prod_{i=1}^n (-t^{-\frac{1}{2}}-t^{\frac{1}{2}})^{nul(T_{i_{S_i}})}\]
	
 By Theorem \ref{thm: states preserved}, $A(S) = A(S_1) + \dots + A(S_n)$ and $B(S) = B(S_1) + \dots + B(S_n)$. Therefore by exponent rules, 
	
	\[t^{-\frac{1}{4}^{A(S)}} = \prod_{i=1}^n t^{-\frac{1}{4}^{A(S_i)}}\]
	
	\[t^{-\frac{1}{4}^{B(S)}} = \prod_{i=1}^n t^{-\frac{1}{4}^{B(S_i)}}\]
	
 By Theorem \ref{thm: nullity preserved}, $nul(T_S) = nul(T_{1_{S_1}}) + \dots + nul(T_{n_{S_n}})$. Again by exponent rules we get 
	
	\[(-t^{-\frac{1}{2}}-t^{\frac{1}{2}})^{nul(T_S)} = \prod_{i=1}^n (-t^{-\frac{1}{2}}-t^{\frac{1}{2}})^{nul(T_{i_{S_i}})}\]

 All together, this gives us the expression: 
	
	\[t^{-\frac{1}{4}^{A(S)}} t^{\frac{1}{4}^{B(S)}}(-t^{-\frac{1}{2}}-t^{\frac{1}{2}})^{nul(T_S)} = \prod_{i=1}^n t^{-\frac{1}{4}^{A(S_i)}} t^{\frac{1}{4}^{B(S_i)}}(-t^{-\frac{1}{2}}-t^{\frac{1}{2}})^{nul(T_{i_{S_i}})}.\]
	
	The right hand side is a term from the sum for the composite knot and the left hand side is the product of the terms corresponding to the same state. Therefore on a term-by-term basis, the large sum is the product of the smaller products as claimed.\\

	\noindent Therefore 
	
	\[V_K = \prod_{i = 1}^n V_{K_i}\]
	
	\noindent as claimed.
\end{proof}


\bibliographystyle{amsalpha}
\bibliography{refs}

\end{document}